\documentclass[11pt]{article}
\usepackage{amssymb}
\usepackage{amsthm}
\usepackage{amsmath}
\usepackage{graphicx}

\parskip=\medskipamount
\def\fpd#1#2{{\displaystyle\frac{\partial #1}{\partial #2}}}
\def\spd#1#2#3{{\displaystyle\frac{\partial^2 #1}
{\partial #2\partial #3}}}
\def\cinfty#1{C^{\scriptscriptstyle\infty}(#1)}
\def\onehalf{{\textstyle\frac12}}
\def\vectorfields#1{{\cal X}(#1)}
\newcommand{\R}{\mathbb{R}}
\def\TMO{T^\circ\!M}
\def\L{{\mathcal L}}

\theoremstyle{plain}
\newtheorem{thm}{Theorem}
\newtheorem{prop}{Proposition}
\newtheorem{definition}{Definition}

\author{T.\ Mestdag\\Department of Mathematics, Ghent University\\
Krijgslaan 281, B--9000 Gent, Belgium\\tom.mestdag@ugent.be}

\title{Finsler geodesics of Lagrangian systems through Routh reduction}

\begin{document}

\date{}
\maketitle

\begin{quote}
{\small {\bf Abstract.} We make use of a symmetry reduction technique called Routh reduction to show that the solutions of the Euler-Lagrange equations of a strongly convex autonomous Lagrangian which lie on a specific energy level can be thought of as geodesics of an associated Finsler function.


}
\end{quote}

\section{Introduction}

In this paper we generalize a theorem from Lagrangian mechanics that is sometimes attributed to Jacobi or referred to as `Maupertuis' Principle'. It may be found in one version or another in a number of standard text books.  Assume that $V(x)$ is a function which is bounded from above, and assume that $g_{ij}(x)$ is a Riemannian metric.
Lagrangians of the type $L= \onehalf g_{ij}{\dot x}^i {\dot x}^j -V$ are often referred to as `simple' Lagrangians. Their energy is a conserved quantity given by the function $E_L = \onehalf g_{ij}{\dot x}^i {\dot x}^j +V$.

\begin{thm}[see e.g.\ \cite{AM,Arnold}] \label{jacobi} Assume $e>V(x)$ for all $x$. The solutions of the Euler-Lagrange equations of the Lagrangian $L= \onehalf g_{ij}{\dot x}^i {\dot x}^j -V$ which have constant energy $E_L=e$ are geodesics of the 'Jacobi metric'
\[
{\tilde g}_{ij} = 2(e-V)g_{ij},
\]
after a reparametrization.
\end{thm}

The above theorem has turned out to be extremely useful in a number of applications, for example in the proof of certain properties about the existence of closed orbits on specific energy levels \cite{Arnold}, or in questions about stability of solutions \cite{Gonzalez}. According to \cite{Koiller} the  Jacobi metric is also an useful tool for studying mechanical systems with nonholonomic constraints.

Below (in Theorem~\ref{mainthm}) we discuss a new technique to obtain a generalization of Theorem~\ref{jacobi} in the context of arbitrary Lagrangians, not necessarily of simple type. The main goal of the paper is to show that one may apply a Lagrangian symmetry reduction method, known as Routh's procedure, for that purpose, together with a homogenization trick.

Routh reduction is a method that takes full advantage of the close relation between symmetries and conserved quantities.  Since Routh's original contribution for cyclic coordinates (see e.g.\ \cite{Routh}) it has been generalized to include Lagrangians that are invariant under a possibly non-Abelian Lie symmetry group in e.g.\ the papers \cite{CM,Edu,Scheurle}. To keep the paper self-contained we develop in Section~\ref{sectwo} a level of generalization that is appropriate for understanding the rest of the paper. The approach we will follow is somewhat different from the one that one may usually find in the literature, in that we focus on the case of a singular Lagrangian, and on the vector fields that define the dynamics (see e.g.\ Proposition~\ref{propthree}). In Section~\ref{secthree} we discuss some aspects of Routh reduction for 1-homogeneous Lagrangians. We show, among other, in Proposition~\ref{prop1} that the Routhian of a 1-homogeneous Lagrangian is again 1-homogeneous, but that the same conclusion does not hold for $k$-homogeneous Lagrangians with $k\geq 2$.

The homogenization method (discussed in e.g.\ \cite{SV}, and in Section~\ref{secfour}) we had mentioned before enables one, roughly speaking, to go back and forth between a time-dependent Lagrangian framework and a Finslerian one. We will show in Theorem~\ref{mainthm} that the role of the Jacobi metric for a non-simple autonomous Lagrangian is played by a Routh-reduced 1-homogeneous function on the tangent bundle which, under appropriate conditions, will be a Finsler function. In Section~\ref{secfive} and Section~\ref{secsix} we discuss some examples and we indicate a few directions for future work.

In the literature, one may find a huge number of generalizations of different aspects of Theorem~\ref{jacobi}. Here we mention only a few papers which seem related to ours. For example, the paper \cite{Yasuda} deals with the case of so-called `magnetic Lagrangian systems'. We will treat this case as an example in Section~\ref{secfive}. Some other works on a Finslerian generalization are the papers \cite{Cheng,Contreras,It}. However, their approach to the issue is clearly different from ours. They use a rather Hamiltonian framework on a cotangent bundle, and they make use of the inverse Legendre transformation of a 2-homogeneous Hamiltonian. Our approach, in contrast, remains throughout distinctly Lagrangian. We use a 1-homogeneous function on a tangent bundle and the Finsler functions we find are explicit. No reference needs to be made to an inverse Legendre transformation.

\section{Routh reduction} \label{sectwo}

Let $M$ be a differentiable manifold, with natural local coordinates $(x^\alpha,y^\alpha)$ on its tangent bundle $\tau_M: TM \to M$. In what follows, we will interpret a dynamical system on $M$ as a vector field. The solutions of the system are then identified with the integral curves of the vector field. For example, a system of autonomous ordinary second-order differential equations, in normal form given by $\ddot x^\alpha = f^\alpha(\dot x, x)$, can be represented by a {\em second-order vector field}, i.e.\ by a vector field $\Gamma$ on $TM$ such that $T\tau_M\circ\Gamma = id$. A second-order vector field is of the form
\begin{equation} \label{sode}
\Gamma = y^\alpha\fpd{}{x^\alpha} + f^\alpha(x,y)\fpd{}{y^\alpha}.
\end{equation}
We will often need to distinguish between regular and singular Lagrangian systems.
\begin{definition}
A Lagrangian $L\in\cinfty{TM}$ is regular if its Hessian by fibre coordinates $(g_{\alpha\beta}) = (\partial^2 L/\partial y^\alpha\partial y^\beta)$ defines a non-singular matrix everywhere. A Lagrangian $L$ is strongly convex if $(g_{\alpha\beta})$ is positive-definite everywhere.
\end{definition}
If the Lagrangian is not regular, we will call it singular. A vector field $\Gamma_L$ will be said to be a {\em Lagrangian vector field} for a  Lagrangian $L$ if it is second-order, and if it satisfies
\begin{equation}\label{ELeq}
\Gamma_L\left(\fpd{L}{y^\alpha}\right)-\fpd{L}{x^\alpha}=0,\quad \forall (x,y) \in TM.
\end{equation}
The base integral curves of $\Gamma_L$ are therefore solutions of the Euler-Lagrange equations for $L$. When $L$ is regular the vector field $\Gamma_L$ is uniquely determined by these conditions; when $L$ is singular, there may exist several Lagrangian vector fields $\Gamma_L$. We will often simply refer to the expressions (\ref{ELeq}) as Euler-Lagrange equations.

A base coordinate is said to be cyclic if the Lagrangian does not depend on it. We explain below how, through so-called Routh reduction, one may eliminate both the coordinates and the velocities corresponding to cyclic coordinates.

Let $L$ be a (possibly singular) Lagrangian function
on (the tangent manifold of) a manifold of the type $M= {\R}^n\times Q$, with coordinates $(x^\alpha)=(x^a,x^i)$. Suppose that the Lagrangian $L$ does not explicitly depend on the coordinates $x^a$ of $\R^n$ or, in other words, that it is invariant under the standard action of $\R^n$ on the first factor of  $T(\R^n\times Q) = \R^n\times \R^n \times TQ$. The invariant function $L$ can then also be interpreted as a function on the quotient $T(\R^n\times Q)/\R^n=\R^n\times TQ$. We will not use a different notation for the interpretation of $L$ as an invariant function on $T(\R^n\times Q)$, or as a function on $\R^n\times TQ$. Likewise, the Lagrangian vector fields $\Gamma_L$ are invariant under the action of $\R^n$, and they reduce to vector fields on $\R^n \times TQ$. We will also keep the same notation, and we write
\[
\Gamma_L = y^i\fpd{}{x^i} + f^a\fpd{}{y^a} + f^i \fpd{}{y^i},
\]
where the functions $f^a,f^i$ do not depend on the coordinates $x^a$.

From the Euler-Lagrange equations for the  coordinates $x^a$,
\[
\Gamma_L\left(\fpd{L}{y^a}\right)-\fpd{L}{x^a}=0,
\]
we see that the functions $\partial{L}/\partial{y^a}$ (interpreted as functions on $\R^n\times TQ$) are first integrals for each of the vector fields $\Gamma_L$. That is to say, we have  conservation laws of the type
\begin{equation} \label{levelset}
\fpd{L}{y^a}=\mu_a
\end{equation}
(with  $\mu_a$  constants) along the solutions of the Euler-Lagrange equations of $L$. The relation (\ref{levelset}) expresses conservation of momentum, in a generalized sense. We will always assume that we can solve relations (\ref{levelset}) for the variables $y^a$, say as $y^a = \iota^a_\mu(x^i,y^i)$. The condition for this locally to happen is that the Lagrangian is $\R^n$-regular.
\begin{definition}
A Lagrangian $L$ on $\R^n\times Q$ is $\R^n$-regular if  the matrix of functions $(\partial^2 L/ \partial y^a \partial y^b)$ is everywhere non-singular.
\end{definition}
Remark that a strongly convex Lagrangian is always $\R^n$-regular.

Under the assumption of $\R^n$-regularity, we may identify the level set of momentum $(\ref{levelset})$ with the tangent manifold $TQ$ (or a suitable open part of it) in the quotient manifold $\R^n\times TQ$. We will assume, from now on, that there exists an injection $\iota_\mu: TQ \to \R^n \times TQ , (x^i,y^i) \mapsto (y^a =\iota_\mu^a(x^i,y^i),x^i,y^i)$ which satisfies the identity,
\begin{equation} \label{id}
\fpd{L}{y^a} \circ\iota_\mu = \mu_a.
\end{equation}

We can now re-write the remaining Euler-Lagrange equations by making use of the so-called modified
Lagrangian function $L^\mu$, also often called {\em the Routhian}. This is the restriction of the function
$L- (\partial L/\partial y^a)y^a$
to the level set (\ref{levelset}) where the momentum is $\mu_a$. Given that the Lagrangian $L$ does not depend on $x^a$, it defines  a function on the tangent manifold of $Q$.
\begin{definition}
The Routhian of an $\R^n$-regular Lagrangian $L$ at the level set where the momentum is $\mu_a$ is the function on $TQ$ given by
\[
L_\mu= (L\circ\iota_\mu) -  \iota^a_\mu \mu_a.
 \]
 \end{definition}
Let us fix a Lagrangian vector field $\Gamma_L$ for $L$. Given that $\Gamma_L$ is tangent to the level set, its restriction to it defines a vector field $\Gamma^\mu_L$ on $TQ$, satisfying $\Gamma_L\circ\iota_\mu = T\iota_\mu\circ\Gamma^\mu_L$. We will often refer to the vector field $\Gamma^\mu_L$ on $TQ$ as the {\em Routhian vector field  of $\Gamma_L$}.

\begin{prop}\label{propthree}
The Routhian vector field $\Gamma_L^\mu$ of a Lagrangian vector field  $\Gamma_L$ on $\R^n\times TQ$ at the momentum level $\mu_a$ is a Lagrangian vector field  $\Gamma_{L_\mu}$ on $TQ$, for the Routhian $L_\mu$.
\end{prop}
\begin{proof}
We need to show that $\Gamma^\mu_L$ is a second-order field on $TQ$ and that it satisfies the Euler-Lagrange equations of $L_\mu$. For the first condition, remark that the second-order property of $\Gamma_L$, when thought of as its reduced vector field on the quotient $\R^n\times TQ$ can be translated as $Tp_1\circ\Gamma_L = p_2$, where $p_1$ and $p_2$ are the natural projections $p_1:\R^n\times TQ \to Q, (y^a,x^i,y^i)\mapsto (x^i)$ and $p_2:\R^n\times TQ \to TQ, (y^a,x^i,y^i)\mapsto (x^i,y^i)$, respectively. Then,
\[
T\tau_Q \circ \Gamma^\mu_L = T(p_1\circ\iota_\mu) \circ\Gamma^\mu_L = Tp_1 \circ \Gamma_L\circ\iota_\mu  = p_2 \circ\iota_\mu = id.
\]

For the second property, we know that, since $\Gamma_L$ and $\Gamma^\mu_L$ are $\iota_\mu$-related, we have that $\Gamma^\mu_L (h\circ\iota_\mu) = \Gamma_L(h) \circ\iota_\mu$, for all functions $h$ on $\R^n\times TQ$. Therefore,
\[
\Gamma^\mu_L \left(  \fpd{L}{y^i} \circ \iota_\mu\right) = \Gamma_L \left(\fpd{L}{y^i} \right ) \circ \iota_\mu = \fpd{L}{x^i}\circ\iota_\mu.
\]
The result then follows from the observation that
\[
\fpd{L_\mu}{x^i} = \fpd{L}{x^i}\circ\iota_\mu, \qquad\mbox{and}\qquad \fpd{L_\mu}{y^i} = \fpd{L}{y^i}\circ\iota_\mu + \left(\fpd{L}{y^a} \circ\iota_\mu\right)\fpd{\iota^a_\mu}{y^i}  - \mu_a \fpd{\iota^a_\mu}{y^i}  = \fpd{L}{y^i}\circ\iota_\mu,
\]
in view of the identity (\ref{id}).
\end{proof}

We can conclude that the solutions $(x^a(t),x^i(t))$ of the Euler-Lagrange equations of $L$ project on those $(x^i(t))$ of $L_\mu$. If we are able to calculate such a solution $(x^i(t))$, we can reconstruct a complete solution $(x^a(t),x^i(t))$ by solving the momentum equations $\dot{x}^a = \iota^a_\mu(x^i,{\dot x}^i)$ for $x^a(t)$.

Remark that we have nowhere assumed that the Lagrangian is regular as a whole, so that this procedure also holds for singular Lagrangians. The only condition we have used so far is that the Lagrangian is $\R^n$-regular.

\section{Homogeneous Lagrangians and Finsler functions} \label{secthree}

Let $\tau:\TMO\to M$ be the slit tangent bundle (tangent bundle with the zero section removed), with coordinates $(x^\alpha,y^\alpha)$. A function $F$ is (positively) 1-homogeneous if $\Delta_M(F)=F$, where $\Delta_M=y^\alpha\partial/\partial{y^\alpha}$ stands for the Liouville vector field on $M$. A spray is a second-order vector field on $\TMO$ for which $[\Delta_M,\Gamma]=\Gamma$.  A spray is therefore a vector field of the form (\ref{sode}) with $\Delta_M(f^\alpha) = 2 f^\alpha$.

Two sprays are said to be projectively equivalent if they have the same base integral curves with given initial point and direction, up to an orientation-preserving reparametrization. A set of sprays with this property is called a projective class of sprays. It is easy to see (see e.g.\ \cite{Szilasibook}) that this condition can be infinitesimally be characterized by the property that, if $\Gamma$ is a member of the class, then so is also the spray $\Gamma+P\Delta_M$ for any function $P$ for which $\Delta_M(P)=P$.

Let $F$ be a 1-homogeneous Lagrangian. We will denote its Hessian with respect to fibre coordinates by $h_{\alpha\beta} = \partial^2 F /\partial y^\alpha \partial y^\beta$. Then, $F$ is singular as a Lagrangian, since $h_{\alpha\beta}y^\beta =0$. We recall a few definitions from e.g.\ \cite{multiplier}.
\begin{definition}
If the kernel of the Hessian $h_{\alpha\beta}$ is exactly span$\langle y \rangle$, we say that $h_{\alpha\beta}$ is quasi-regular. The Hessian  $h_{\alpha\beta}$ is positive quasi-definite if $h_{\alpha\beta}w^\alpha w^\beta \geq 0$, with equality only when $w^\alpha=\lambda y^\beta$.
\end{definition}

Assume that $\Gamma_F$ is a Lagrangian vector field for $F$. Any other Lagrangian field $\Gamma_F + V^\alpha \partial /\partial y^\alpha$  satisfies
\begin{equation}\label{other}
V^\alpha h_{\alpha\beta}=0.
\end{equation}
In particular, if $\Gamma_F$ was a spray, then any other member of its projective class will also be a Lagrangian vector field for $F$. When the Hessian of $F$ is quasi-regular, we know that $h_{\alpha\beta}w^\beta = 0$ if and only if $w^\beta = \lambda y^\beta$ for some constant $\lambda$. In that case, we can conclude from relation (\ref{other}) that any other Lagrangian vector field is of the type $\Gamma_F + P(x,y) \Delta$, for some  function $P$. If, in addition to being quasi-regular, $F$ is positive everywhere we can conclude from Proposition~9.1.30 of \cite{Szilasibook} that the Hessian of $E=\onehalf F^2$ is non-degenerate, and that $E$ can therefore be though of as a regular Lagrangian. In that case, there exists a unique (canonical) spray $\Gamma_E$ that satisfies,
\[
 \Gamma_E\left( \fpd{E}{y^\alpha}  \right) - \fpd{E}{x^\alpha} = 0,
\]
the Euler-Lagrange equation for $E=\onehalf F^2$. Since $0=\Gamma_E(E) =  F\Gamma_E(F)$ (which expresses conservation of energy, valid for any autonomous Lagrangian) it is easy to see that $\Gamma_E$ also satisfies the Euler-Lagrange equations for $F$. The projective equivalence class of Lagrangians sprays of $F$ coincides, in that case, with the one of  $\Gamma_E$. The spray $\Gamma_E$ represents within that class the one for which the geodesics are parametrized by constant arc-length $F$.

\begin{definition} \label{deffour}
A Finsler function is a smooth function $F$ on $\TMO$
which is  positive, 1-homogeneous, and which is such that the matrix $(g_{\alpha\beta})= \left(\onehalf \partial^2 {F^2}/\partial{y^\alpha}\partial{y^\beta}\right)$ is everywhere positive-definite (i.e.\ $E=\onehalf F^2$ is strongly convex).
\end{definition}

We refer to e.g.\ \cite{BCS,Szilasibook} for more generalities on Finsler and spray geometry. The following result can be found in \cite{multiplier}.
\begin{prop} \label{multprop}
Let $F$ be a 1-homogeneous function such that its Hessian $h_{\alpha\beta}$ is positive quasi-definite. Then for any $x_0 \in M$ there is a neighbourhood $U$ of $x_0$ in $M$  and a function $\tilde F$ defined on $TU$ such that $\tilde F$ is a Finsler function which differs from $F$ by a total derivative.
\end{prop}
Obviously, the geodesics of $F$ and $\tilde F$ remain the same after a change by a total time derivative.

Let us consider now the case where $M=\R^n\times Q$ as before, but with a 1-homogeneous Lagrangian $L=F$.
\begin{prop} \label{prop1}
The Routhian $F_\mu$ of a 1-homogeneous $\R^n$-regular Lagrangian $F$ on $\R^n\times TQ$ is a 1-homogeneous function on $TQ$.
\end{prop}
\begin{proof}
The Routhian $F_\mu = (F \circ\iota_\mu)-\mu_a\iota_\mu^a$ on $TQ$ will be 1-homogeneous if $F$ is, and if the functions $\iota^a_\mu$ are.
Since $F$ is 1-homogeneous, we have, among other, that
\begin{equation} \label{hom}
y^ih_{ai} = - y^b h_{ab}.
\end{equation}
Let $\Delta_Q = y^i \fpd{}{y^i}$. By taking a $\Delta_Q$-derivative of the identity (\ref{id}) we get
\[
0 = y^i\left( h_{ai}\circ\iota_\mu\right) + \left(h_{ab}\circ\iota_\mu\right) \fpd{\iota^b_\mu}{y^i}y^i =(h_{ab} \circ \iota_\mu) \left( -\iota^b_\mu + \Delta_Q(\iota_\mu^b) \right).
\]
Due to the assumed $\R^n$-regularity, we can conclude that $\Delta_Q(\iota^b_\mu)=\iota^b_\mu$.
\end{proof}

Remark that the same conclusion does not hold for a $k$-homogeneous lagrangian with $k\geq 2$ (such as, for example, $F=\onehalf g_{\alpha\beta}y^\alpha y^\beta$), because of the appearance of extra terms in expression (\ref{hom}).

\begin{prop} \label{propsix} Assume that $F$ is 1-homogeneous and $\R^n$-regular. If the Hessian of $F$ is positive quasi-definite, then so is the Hessian of $F_\mu$.
\end{prop}
\begin{proof}
We have
\[
\spd{F_\mu}{y^i}{y^j}w^iw^j = \left(h_{ij} \circ \iota_\mu \right)w^iw^j + \left( h_{ia}\circ\iota_\mu\right) \fpd{\iota^a_\mu}{y^j}w^iw^j
\]
After taking a $\partial/\partial y^i$-derivative of the identity (\ref{id}), we get that $ \partial{\iota^a_\mu}/\partial{y^j} = - h^{ab}h_{bj}$. With this, we can write the above as
\begin{eqnarray*}
\spd{F_\mu}{y^i}{y^j}w^iw^j &=& (h_{ij}\circ \iota_\mu)w^iw^j - 2 (h_{ia}\circ \iota_\mu) w^i (h^{ab}\circ \iota_\mu)(h_{bj}\circ \iota_\mu)w^j \\&&+  (h_{ia}\circ \iota_\mu) (h^{ab}\circ \iota_\mu) (h_{bj}\circ \iota_\mu) w^iw^j \\
&=& (h_{\alpha\beta}\circ \iota_\mu)W^\alpha W^\beta
\end{eqnarray*}
with $W^i=w^i$ and $W^a = - (h^{ab}\circ \iota_\mu)(h_{bj}\circ \iota_\mu)w^j$. Since the right-hand side is always positive or zero, so is the left-handside. The right hand side can only be equal to zero if the  $W^\alpha=\lambda y^\alpha$, but then $w^i=\lambda y^i$ and and this is not in conflict with $W^a =\lambda y^a$, in view of relation (\ref{hom}).
\end{proof}
In view of Proposition~\ref{multprop}, we can conclude from Proposition~\ref{prop1} and \ref{propsix} that, in case  the Hessian of $F$ is positive quasi-definite, we may add a suitable total time derivative to $F_\mu$ to obtain a local Finsler function.

\begin{prop}\label{proptwo}
The Routhian vector field of a Lagrangian spray $\Gamma_F$ of a 1-homogeneous and $\R^n$-regular Lagrangian $F$ is again a spray.
\end{prop}
\begin{proof}
The vector field $\Gamma_F$, now thought of as a vector field on $T(\R^n\times Q)$, and the vector field $\Delta_M = y^\alpha \partial / \partial y^\alpha$ satisfy the bracket relation $[\Gamma_F,\Delta_M] =\Gamma_F$. If we set $p_3: \R^n \times Q \to Q$, then $\Gamma_F$ and $\Delta_M$ are $p_3$-related to, respectively, the Routhian vector field $\Gamma_{F_\mu}$ and the vector field $\Delta_Q = y^i \partial /\partial y^i$, and therefore so are their brackets.
\end{proof}

Remark that, for example in the case where the Hessian of $F$ is positive quasi-definite, the canonical spray of $F_\mu$ (with $E_\mu = \onehalf F^2_\mu$) is not the Routhian vector field  $\Gamma^\mu_E$ of the canonical spray $\Gamma_E$ of $F$, but its projective equivalent $\Gamma^\mu_E +P\Delta_Q$, where the projective factor is  $P=- (\mu_a \Gamma_E^\mu(\iota^a_e))/F_\mu$. This is in agreement with the fact that the Routhian of $E$ is no longer a 2-homogeneous function.

\section{Time-dependent Lagrangians and autonomous Lagrangians} \label{secfour}

We now say a few words about time-dependent Lagrangian systems (see e.g.\ \cite{LR} for more details).

Let  $\L \in \cinfty{\R\times TQ}$ be a {\em time-dependent} regular Lagrangian. If we use coordinates $(t,x^i,v^i)$ on $\R\times TQ$, then regularity means here `$TQ$-regularity', in the sense that the Hessian matrix $(\partial^2 \L/\partial v^i\partial v^j)$ is non-singular everywhere. The Euler-Lagrange equations of $\L$ can be thought of as being given by the integral curves the unique {\em time-dependent second-order vector field}
\[
\Gamma_\L = \fpd{}{t} + v^i\fpd{}{x^i} + f^i(t,x,v)\fpd{}{v^i}
\]
that is determined by
\[
\Gamma_\L\left( \fpd{\L}{{v}^i}  \right) - \fpd{\L}{x^i} = 0.
\]

We recall  from e.g.\ \cite{SV} that there exists a certain equivalence between time-dependent Lagrangians and 1-homogeneous Lagrangian functions. The `homogenization trick' relies on augmenting the dimension of the configuration space by one, by adding an extra `time-velocity' to the picture.  With each time-dependent Lagrangian we may  define a 1-homogeneous function $F$ by setting
\begin{equation} \label{defF}
F(x^0,x^i,y^0,y^i) = y^0\L (x^0,x^i,y^i/y^0).
\end{equation}
This function is only well-defined for $y^0\neq 0$. For most of what we say below, we will restrict our attention to vectors with  $y^0>0$. We will use the notation $T^\circ\!(\R\times Q)$ for this set.

The manifold $\R\times TQ$ may be identified with the submanifold $y^0=1$ in $T^\circ\!(\R\times Q)$. If we restrict $F$ to it, we get back the time-dependent Lagrangian $\L$. In what follows, we will use the notations $\pi: T^\circ\!(\R \times Q) \to \R \times TQ$ for the map $(x^0,x^i,y^0,y^i) \mapsto (t=x^0,x^i=x^i,v^i = y^i/y^0)$ and $j: \R\times TQ \to T^\circ\!(\R \times Q)$ for $(t,x^i,v^i) \mapsto (x^0=t,x^i=x^i,y^0=1,y^i=v^i)$. Then $\pi\circ j = id$ (but not the other way round), $\L=F\circ j$ and $F = \L \circ \pi$. The following relations hold between their derivatives:
\begin{eqnarray*}
&&\fpd{F}{x^\alpha}= \fpd{\L}{x^\alpha}\circ \pi,\qquad
\fpd{F}{y^i}  = \fpd{\L}{v^i}\circ\pi,\\ && \fpd{F}{y^0}= (\L\circ\pi)-\frac{y^i}{y^0}\left(\fpd{\L}{v^i}\circ\pi\right).
\end{eqnarray*}

Since $F$ is singular as a Lagrangian, there exists many second-order vector fields on $T^\circ\!(\R\times Q)$ which satisfy its Euler-Lagrange expressions. Let us assume for now that $E=\onehalf F^2$ defines a regular Lagrangian (we show in Proposition~\ref{propfive} that this is a quite natural assumption). Let us choose the unique spray in the projective class of $\Gamma_E$ that is tangent to the submanifold $y^0=1$, i.e.\ $\Gamma_{F,y^0} = \Gamma_E +  P \Delta$ with $P=-\Gamma_E(y^0)/y^0$. A version of the next proposition can also be found in \cite{SV}. Let us denote by $\Gamma_t$ the restriction  of $\Gamma_F$ to $y^0=1$. $\Gamma_t$ is the vector field  on $\R\times TQ$ that is $j$-related to $\Gamma_F$.

\begin{prop}\label{propfour} If $\Gamma_{F,y^0}$ is the spray that is tangent to the submanifold $y^0=1$, its restriction $\Gamma_t$ to $y^0=1$ can be identified with the time-dependent Lagrangian vector field $\Gamma_\L$ of $\L$ on $\R\times TQ$.
\end{prop}
\begin{proof}
Since we know that $\Gamma_t (h\circ j ) = \Gamma_F(h) \circ j$, for all functions $h$ on $T(\R\times Q)$. Therefore
\[
\Gamma_t \left(  \fpd{\L}{v^i}  \right)  =
\Gamma_t \left(  \fpd{F}{y^i} \circ j \right) = \Gamma_F \left(\fpd{F}{y^i} \right) \circ j = \fpd{\L}{x^i} \circ (\pi\circ j) = \fpd{\L}{x^i}.
\]
Moreover, $\langle\Gamma_t,dt\rangle = j^* \langle\Gamma_F,dx^0\rangle = j^*y^0=1$, and the statement follows.
\end{proof}

Fixing one particular element of the projective class of sprays, in effect, fixes a certain parametrization for the geodesics of $F$.

Suppose now that we are given an {\em autonomous} Lagrangian $L$ on $TQ$. We may think of it as a function $\L=L\circ p$ on $\R\times TQ$ (with $p: \R \times TQ\to TQ$) which so happens not to depend on time $t$ explicitly. If we assume that $L$ is regular, there is only one Lagrangian vector field $\Gamma_L$ on $TQ$. Also $\L$ is regular in that case and the relation between $\Gamma_L$ and the time-dependent Lagrangian vector field on $\R\times TQ$ can simply be written as $\Gamma_\L = \partial/\partial t + \Gamma_L$. In what follows we will no longer make a notational distinction between $\L$ and $L$. When necessary, however, we will use different notations for the corresponding vector fields $\Gamma_L$ (on $TQ$) and $\Gamma_\L$ (on $\R\times TQ$).

It is well-known that for an autonomous Lagrangian, $\Gamma_\L$ (and $\Gamma_L$) exhibits an energy first integral $E_L = \Delta_Q(L)-L$. If we now use the homogenization trick to construct the 1-homogeneous function $F$ as above in (\ref{defF}), we easily see that the coordinate $x^0$ is cyclic for $F$. Therefore there is conservation of the momentum, of the type
\begin{equation}\label{homenergy}
\fpd{F}{y^0}(x^a,y^a) = -e
\end{equation}
(the minus sign is chosen for later convenience). In order to proceed with a  Routh-type reduction (as in the previous sections) we must be able to re-express the above momentum relation
in the form $y^0 = \iota_e(x^i,y^i)$. The condition for this locally to happen is that $F$ is $\R$-regular, which means here that  $\partial^2 F/\partial (y^0)^2 \neq 0$.
\begin{prop} \label{propfive}
For an autonomous strongly convex Lagrangian $L$, the homogeneous function $F$ in (\ref{defF}) is $\R$-regular and the Hessian of $F$ is positive quasi-definite in $T^\circ\!(\R\times TQ)$.
\end{prop}
\begin{proof} Since
\[
\frac{\partial^2 F}{\partial (y^0)^2} = \frac{1}{(y^0)^3} \left(\spd{L}{{v}^i}{{v}^j} \circ \pi \right) y^iy^j,
\]
the first statement follows. For the second, we get for $(w^\alpha) =(w^0,w^i)$ that
\[
\spd{F}{y^\alpha}{y^\beta} w^\alpha w^\beta = \frac{1}{y^0}\left(\spd{L}{v^i}{v^j}\circ\pi\right) \left( w^i - \frac{y^i}{y^0}w^0\right)   \left( w^j - \frac{y^j}{y^0}w^0\right) \geq 0,
\]
with equality if $w^i = (w^0/y^0) y^i$. Since always $w^0 = (w^0/y^0) y^0$, we can conclude from this that indeed $w^\alpha = \lambda y^\alpha$, with $\lambda=w^0/y^0$.
\end{proof}

We will assume from now on that $L$ is strongly convex. Remark that the momentum relation (\ref{homenergy}) in $T^\circ\!(\R\times TQ)$ is here of the form $e = E_L(x^i, y^i/y^0)$. The sought (local) function $\iota^0_e(x^i,y^i) > 0$ satisfies therefore the identity
\begin{equation} \label{homenergy2}
e = E_L(x^i, y^i/\iota^0_e(x^i,y^i)).
\end{equation}
From Proposition~\ref{prop1} we may conclude that $\iota_e$ is always a 1-homogeneous function, and so is also the Routhian of $F$, which in this case becomes
\begin{equation} \label{Fe}
F_e(x^i,y^i) = (F - \fpd{F}{y^0}y^0 )|_{y^0=\iota^0_e} = \iota^0_e(x^i,y^i)\left(L\left(x^i,\frac{y^i}{\iota^0_e(x^i,y^i)}\right) + e\right).
\end{equation}
From Propositions~\ref{propsix} and \ref{propfive} we know that the Hessian of $F_e$ is also positive quasi-definite, and that $F_e$ can be suitably changed into a local Finsler function (in view of Proposition~\ref{multprop}). In the section on examples, we will discuss some cases when $F_e$ defines a global Finsler function, or when it can be changed into one. We will use the notation $E_e = \onehalf F_e^2$ for its energy function.

We know from Proposition~\ref{propthree} that the restriction of a Lagrangian vector field  $\Gamma_F$ on $\R\times TQ$ to the level set where the momentum $\partial F/\partial y^0$ is $-e$ is a Lagrangian vector field $\Gamma_{F_e}$ on $TQ$, for the Routhian $F_e$. We have that  $\Gamma_F\circ\iota_e = T\iota_e\circ\Gamma_{F_e}$. Since $F_e$ is locally a Finsler function (after possibly adding a total time derivative) we know that there exists a whole projective class of sprays, which all have the same geodesics as points sets. Picking one out of the class, is equivalent with choosing a parametrization for the geodesics. If we take the canonical spray $\Gamma_{E_e}$, the parametrization is given by arc length. Another choice is the following. Consider again the unique Lagrangian spray of $F$ that is tangent to $y_0=1$, $\Gamma_{F,y_0}$. Its restriction to $\partial F/\partial y_0=-e$ defines, by Proposition~\ref{propthree} and \ref{proptwo} a specific Lagrangian spray of the projective class of $F_e$. Let's call it ${\bar\Gamma}$ for now. Since this vector field is $\iota_e$-related to $\Gamma_{F_e,y^0}$, we have that ${\bar\Gamma}(\iota_e^0) = {\bar\Gamma}(y_0\circ \iota_e) = \Gamma_{F,y^0}(y^0)=0$. We can therefore conclude that $\bar\Gamma = \Gamma_{F_e,\, \iota_e^0}$, the unique vector field in the projective class of $F_e$ that is tangent to $\iota_e^0=1$. Remark that the relation $\iota_e^0(x,y)=1$ is equivalent with $E_L(x,y) =e$, in view of relation (\ref{homenergy2}). Given that, according to Proposition~\ref{propfour},  $\Gamma_\L$ was the restriction of the same vector field $\Gamma_{F,y^0}$ to $y^0=1$, we get, after discarding the term $\partial/\partial t$ in $\Gamma_\L$, that $\Gamma_L \mid_{E_L=e} = \Gamma_{F_e,\,\iota_e^0} \mid_{\iota_e^0=1}$ as vector fields on $TQ$. In conclusion, we can say that:

\begin{thm}\label{mainthm} Let $L$ be a strongly convex Lagrangian on $TQ$.

\begin{enumerate} \item The restriction of its Lagrangian vector field $\Gamma_L$ to the energy level set where $E_L=e$ coincides with the restriction to $\iota_e^0=1$ of the unique spray that is both projectively equivalent to the canonical spray $\Gamma_{E_e}$ of the function $F_e$ in (\ref{Fe}) and tangent to $\iota^0_e=1$.

 \item The solutions of the Euler-Lagrange equations of $L$ with energy $E_L=e$ are geodesics of a local Finsler function $F_e$ (possibly after the addition of a total time derivative, possibly after a reparametrization).
  \end{enumerate}
\end{thm}
In the statement of the theorem, we have used the word `geodesic' in the sense of a set of points. The addition `possibly after a reparametrization' refers to the fact that one may use either the canonical spray of $F_e$, or any other spray (like $\Gamma_{F_e,\,\iota^0_e}$) in its projective class to compute those geodesics as parametrized curves.

\section{Examples} \label{secfive}

{\bf Example 1.} Consider first a positive, strongly convex and $k$-homogeneous Lagrangian $L$ on $Q$, $\Delta_Q(L) = pL$, where $k\geq 2$. The defining relation (\ref{homenergy2}) for $\iota_e^0$ is then $(k-1)L(x,y/\iota^0_e) =e$. It has $\iota^0_e = ((k-1) L / e)^{1/k}$ as a solution, when $e> 0$. With this, the corresponding 1-homogeneous function is
\[
F_e = k\left( \frac{k-1}{e} \right)^{\frac{1-k}{k}} L^{\frac{1}{k}}.
\]
Since here $E_L = (k-1)L$, we get that $\Gamma_L(L) = 0$. One can verify that the Lagrangian field $\Gamma_L$ of a $k$-homogeneous Lagrangian is in fact a spray, see e.g.\ \cite{LR}. Since $L$ is also proportional with a constant factor to $(\iota_e^0)^k$, and since we assume that $\Gamma_{F_e}$ is such that $\Gamma_{F_e,\,\iota^0_e}(\iota^0_e)=0$, the sprays $\Gamma_L$ and $\Gamma_{F_e,\,\iota^0_e}$ will coincide. This is, in particular, the case when $k=2$ and $L$ is the kinetic energy associated to a Riemannian metric.

{\bf Example 2.} We will call a Lagrangian {\em magnetic} if it is of the type
\[
L(x,v)=\onehalf g_{ij}(x)v^iv^j+\beta_i(x)v^i - V(x),
\]
where $g$ is a Riemannian metric on $Q$, $\beta$ is a 1-form on $Q$ and $V$ is a smooth function on $Q$. The Lagrangian $L$ is always strongly convex, since its Hessian is given by the Riemannian metric. Under the assumption that $V(x)$ is bounded from above and $e> \max_{x\in Q} V(x)$, one easily finds that
\[
\iota^0_e(x,y) = \sqrt{\frac{g_{ij}(x)y^iy^j}{2(e-V(x))}}
\]
is the positive solution of the momentum relation (\ref{homenergy2}).  With that the function $F_e$ is of so-called Randers type,
\begin{equation} \label{Randers}
F_e(x,y)=\sqrt{{\tilde g}_{ij}(x)y^iy^j}+{\beta}_i(x)y^i,
\end{equation}
where ${\tilde g}_{ij} = 2(e-V)g_{ij}$ is the {\em Jacobi metric} we mentioned in the Introduction.
Obviously, when $\beta_i=0$, we recover the statement of Theorem~\ref{jacobi}, when we consider the geodesics as being given by the integral curves of the canonical spray $\Gamma_{E_e}$, where $E_e =\onehalf F_e^2 =\onehalf {\tilde g}_{ij}y^iy^j$.

The case where $V=0$ can be found in \cite{Yasuda}, but the method used to obtain the result is different from ours. Remark that in two- or three-dimensional
Euclidean space this is the Lagrangian for the motion of a classical charged particle, of unit charge, in the magnetic field determined by
$d(\beta_idx^i)$.

Let us consider again the general case. From the statement in Theorem~\ref{mainthm} we know that strong convexity only guarantees the existence of a local Finsler function. For this example it is possible to give a criterion for when the Randers function (\ref{Randers}) is actually a global Finsler function.  Let $\beta^\#$ be the vector field which is such that $g(\beta^\#,X) = \beta(X)$ for all $X\in\vectorfields{Q}$. It is well-known that a Randers function determines a global Finsler function in the region determined by those $x\in Q$ for which $\tilde g(\beta^\#,\beta^\#)<1$  (see e.g.\ \cite{BCS}). One easily verifies that $F_e$ will be a global Finsler function if and only if $e> \onehalf g^{ij}(x)\beta_i(x)\beta_j(x) + V$ for all $x\in Q$, or if and only if
\[
e> \max_{x\in Q}\left(\onehalf g^{ij}\beta_i\beta_j + V\right).
\]
We can actually push this limit even more down. Recall from Proposition~\ref{multprop} that there remains some freedom that we have not exploited so far. By adding a total time-derivative (i.e.\ by replacing the one-form $\beta$ by $\beta-df$ for some function $f$) we do not change the set of geodesics. With that in mind, the above criterion can be most easily rewritten by invoking the  Hamiltonian function of $L$, i.e.\ the function on $T^*Q$ given by $H(x,p)=\onehalf g^{ij}(x)(\beta_i(x)-p_i)(\beta_j(x)-p_j) + V(x)$. If the number
\[
c(L) = \inf_{f\in\cinfty{Q}} \max_{x\in Q} H (x, df)
\]
exists, we may conclude that the function $F_e + (\partial f/\partial x^i)y^i$ (with + sign in $F_e$) is a global Finsler function, for each energy level $e$ such that $e> c(L)$. The number $c(L)$ coincides, at least for the case of a magnetic Lagrangian, with what is called  Ma\~n\'e's critical value in \cite{Gonzalez}.

{\bf Example 3.} Let us consider the following specific two-dimensional Lagrangian of magnetic type,
\[
L(x,v)= \frac{v_1^2+v_2^2}{16\left(1-x_1^2-x_2^2 \right)^2}  +  \frac{x_2v_1-x_1v_2}{2( 1-x_1^2-x_2^2)  }.
 \]
This Lagrangian is strongly convex only in the area $Q= \{ 1-x_1^2-x_2^2>0 \}$, i.e.\ in the open unit disk. If we assume that $e>0$ the function $\iota^0_e$ satisfies
\[
(\iota^0_e(x,y))^2 = \frac{ y_1^2+y_2^2 }{16e(1-x_1^2-x_2^2)^2}.
\]
If we introduce a new parameter $\tau = 1/ \sqrt{e}$, the corresponding homogeneous function (\ref{Fe}) is, up to a constant factor,
\begin{equation} \label{Ftau}
F_\tau(x,y)= \frac{\sqrt{y_1^2+y_2^2}+ \tau(x_2y_1-x_1y_2)}{2(1-x_1^2-x_2^2)}.
\end{equation}

Different values of the parameter $\tau$ correspond with different energy levels for the original Lagrangian $L$. The above (parametrized) Randers function $F_\tau$ has also appeared in our paper \cite{horo}, where it was shown that all of its geodesics are circles (straight lines are considered as circles of infinite radius). If we set $\tau=0$ we get the (Finsler function associated to the) Riemannian metric of the hyperbolic plane (in its version as the Poincar\'e disk). Its geodesics are circles that cross the border of the unit disk perpendicularly. If we set $\tau=1$ we get the so-called horocycles as the Finslerian geodesics (actually, also when we would take $\tau=-1$ if that were allowed in the current setting). These are circles that are tangent to the unit circle at the origin. The configurations for other values of $\tau$ interpolate between, and extend beyond, these.

In each of the two figures below, we have chosen a fixed initial position and a fixed initial tangent direction. We have plotted for a few values of $\tau\geq 0$ the two geodesics of $F_\tau$ corresponding to each of the two orientations of the initial tangent line. The black circles represent the special cases of a hyperbolic geodesic or the horocycles (one for each orientation). The Randers function (\ref{Ftau}) will be a Finsler function if $F_\tau>0$, or $x_1^2+x_2^2<\min(1,e=\tau^{-2})$. Thus for $0\leq\tau\leq 1$, $F_\tau$ is a globally defined Finsler metric over the whole of the open unit disk. Its geodesics are the grey circles outside the horocycles (to be precise: the arcs of those circles that lie within the open unit disk). But for $\tau>1$ the function $F_\tau$ is a Finsler function only on $x^2+y^2<e$. Its geodesics are given by grey circles inside the horocycles.

\begin{figure}[h]
\hspace*{2cm}\begin{tabular}{cc}
\begin{minipage}{7cm}
\includegraphics[scale=0.5]{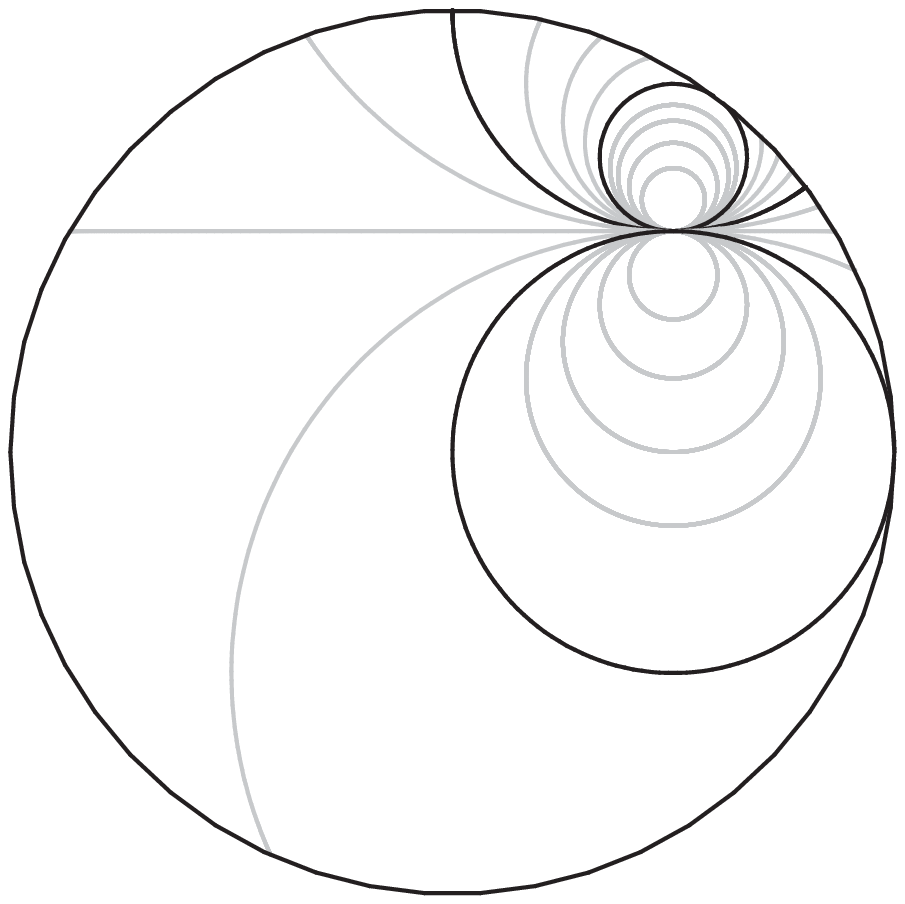}
\end{minipage}
&
\begin{minipage}{7cm}
\includegraphics[scale=0.5]{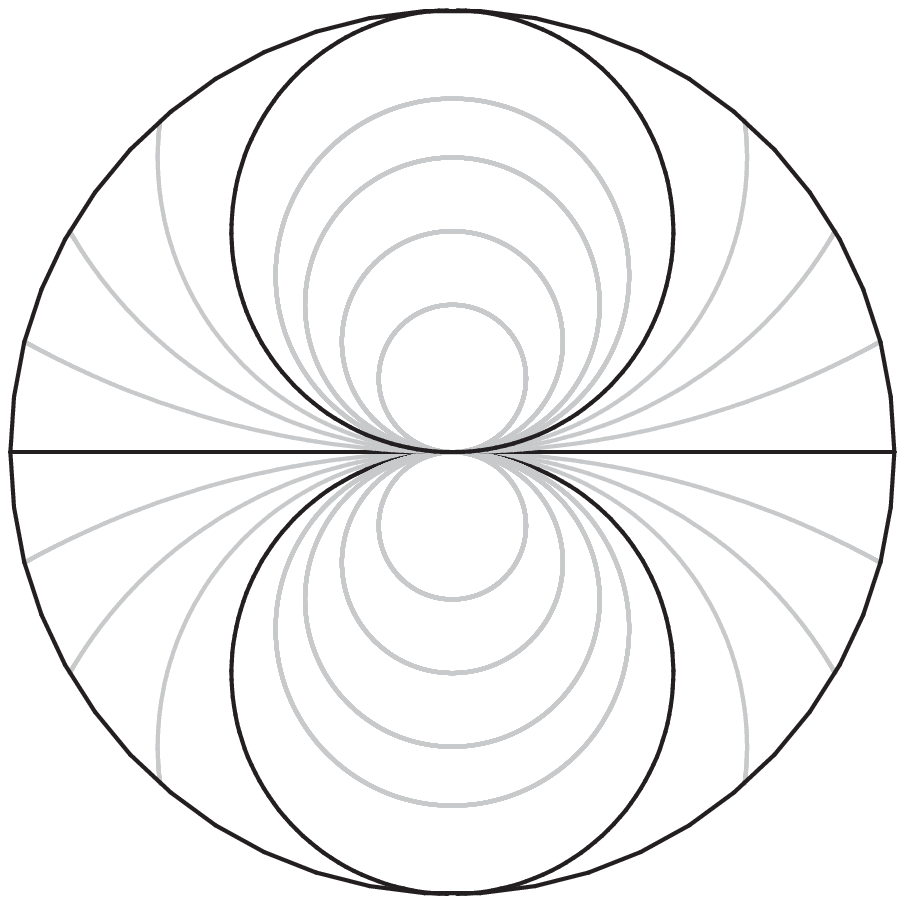}
\end{minipage}
\end{tabular}
\end{figure}

\section{Outlook}\label{secsix}

The last example was motivated in \cite{horo} by a study of the so-called {\em Finsler projective metrizability problem}, see e.g.\ \cite{Buca,multiplier,Szilasibook}. This is, roughly speaking, the question whether or not the unparametrized base integral curves of a given geodesic class of sprays are the geodesics of a yet to be determined Finsler function. It would be of interest to see whether Theorem~\ref{mainthm} can provide some help in finding such a Finsler function, or some new insight. An indication that this is likely going to be the case, may be derived from the following observation. It is well-known that here exists a canonical torsion-free invariant linear connection on each Lie group, given by $\nabla_X Y = \onehalf [X,Y]$ for left-invariant vector fields $X,Y$. Its canonical spray determines a projective class of sprays on each Lie group, in a natural way. In a series of papers, starting with \cite{Thompson}, G.\ Thompson and co-workers have studied the question whether there exists a regular Lagrangian such that its Euler-Lagrange equation are equivalent with the geodesic equations of the connection. Since the Lagrangian field of a connection is a spray, Theorem~\ref{mainthm} will provide a (local) Finsler function for each of these Lagrangians, which will give a solution for the metrizability problem of the canonical connection.

Much of the method we have developed in this paper relied on the fact that the time variable $t$ generates, for an autonomous Lagrangian $L$, a cyclic variable $x^0$ for the corresponding 1-homogeneous function $F$. The conserved momentum we have used in Routh's procedure was related to conservation of energy. We may translate this to the more general context of Noether symmetries. It is well-known energy conservation is a consequence of the presence of the Noether symmetry $\partial/\partial t$ for each autonomous Lagrangian. There exist, however, autonomous Lagrangian systems with more general time-dependent Noether symmetries, not necessarily of the type of a cyclic coordinate. Just to give one example, consider the Lagrangian given by $L=\onehalf (v_1^2 + v_2^2)$ and its symmetry given by
\[
2t\fpd{}{t} + x_1 \fpd{}{x_1} + x_2 \fpd{}{x_2}.
\]
The corresponding conservation law is $(v_1^2 + v_2^2)t + x_1v_1+x_2v_2=\mu$. One may show that any time-dependent Noether symmetry of a possibly time-dependent Lagrangian $\L$ defines an autonomous Noether symmetry of the function $F$ in (\ref{defF}). Moreover, as we have stated before, also Routh's reduction method remains valid in a more general context.  It would be of interest to investigate how much of our method can be generalized to include also these types of Noether symmetries and their conservation laws.

{\bf Acknowledgements.} This work is part of the {\sc irses} project {\sc geomech} (nr. 246981) within the 7th European Community Framework Programme.

\end{document}